\providecommand\reserveinserts[1]{}
\documentclass[AMS,Times1COL]{WileyNJDv5}
\makeatletter
\def\oddhead@titlepage@info{}
\def\evenhead@titlepage@info{}
\def\oddfoot@titlepage@info{}
\def\evenfoot@titlepage@info{}
\makeatother
\hypersetup{
  pdftitle={Multiplier obstructions for Legendre pairs of length 333},
  pdfauthor={Arthur F. Ramos, David B. Hulak, Ruy J. G. B. de Queiroz}
}

\articletype{Research Article}

\journal{Journal of Combinatorial Designs}
\volume{00}
\copyyear{2026}
\startpage{1}

\begin{document}

\title{Multiplier obstructions for Legendre pairs of length 333}

\author[1]{Arthur F. Ramos}
\author[2]{David B. Hulak}
\author[3]{Ruy J. G. B. de Queiroz}

\authormark{RAMOS \textsc{et al.}}
\titlemark{MULTIPLIER OBSTRUCTIONS FOR LEGENDRE PAIRS}

\address[1]{\orgname{Microsoft}, \country{USA}}
\address[2]{\orgname{Independent Researcher}}
\address[3]{\orgdiv{Centro de Inform\'atica}, \orgname{Universidade Federal de
Pernambuco}, \country{Brazil}}

\corres{Arthur F. Ramos, Microsoft, USA. \email{arfreita@microsoft.com}}

\abstract[Abstract]{%
A Legendre pair of length \(333\) would yield a Hadamard matrix of order
\(668\), the smallest order presently unresolved by the Hadamard conjecture.
We study the structured case in which both sequences are fixed by a common
subgroup \(H\leq (\mathbb Z/333\mathbb Z)^\times\) acting by coordinate
multiplication. We prove that such a pair can exist only when \(|H|\leq 6\).
After a mod-\(3\) compression reduces the problem to an order-\(108\) kernel,
there are exactly \(30\) subgroups. We exclude \(21\) of them,
including all \(19\) subgroups of order at least \(9\). The final order-\(9\)
subgroup is eliminated analytically: its orbit structure restricts the
\(9\)-compressed entries to
\(\{\pm1,\pm17,\pm19,\pm35,\pm37\}\); the Legendre equations force a
\(+17,-17\) pair in one compressed sequence, and a single-shift
autocorrelation bound then contradicts the required compressed correlation.
The remaining exclusions use full-image compression, a row-sum congruence,
exact meet-in-the-middle enumeration, and proof-carrying pseudo-Boolean
encodings. The solver-assisted cases are accompanied by independently checked
DRAT proofs or direct arithmetic certificates. The result constrains fixed
common-multiplier symmetry only; the unrestricted existence problems remain
open.}

\keywords{Hadamard matrix, Legendre pair, supplementary difference set,
multiplier, periodic autocorrelation, proof certificate}

\maketitle

% Wiley's class reserves \doi for production metadata, which suppresses DOI
% text emitted by the bibliography style. Restore a printable linked form.
\renewcommand{\doi}[1]{\href{https://doi.org/#1}{\nolinkurl{doi:#1}}}

\renewcommand\thefootnote{}
\footnotetext{\textbf{2020 Mathematics Subject Classification:}
05B20 (primary); 05B10, 11B83 (secondary).}
\footnotetext{\textbf{ORCID:} Arthur F. Ramos
\url{https://orcid.org/0009-0003-3568-0325}; David B. Hulak
\url{https://orcid.org/0009-0002-8056-1774}; Ruy J. G. B. de Queiroz
\url{https://orcid.org/0000-0003-1482-0977}.}
\renewcommand\thefootnote{\fnsymbol{footnote}}
\setcounter{footnote}{1}

\section{Introduction}\label{sec:introduction}

A Hadamard matrix of order \(n\) is a matrix \(M\in\{\pm1\}^{n\times n}\)
satisfying \(MM^{\mathsf T}=nI_n\) \cite{Hadamard1893}. Hadamard's conjecture,
often associated with Paley's foundational constructions \cite{Paley1933},
predicts that such a matrix exists whenever \(4\mid n\). All admissible orders
below \(668\) are known, while order \(668\) remains unresolved
\cite{KharaghaniTayfehRezaie2005,CatiPasechnik2024}. A recent construction of
a \(64\)-modular Hadamard matrix at this order provides modular, but not
integer, orthogonality \cite{Eliahou2025}.

Legendre pairs provide one of the most direct structured routes to this open
order. Fletcher, Gysin, and Seberry showed that a Legendre pair of odd length
\(L\) yields a Hadamard matrix of order \(2L+2\)
\cite{FletcherGysinSeberry2001}. Consequently, length
\[
L=333=3^2\cdot37
\]
is the relevant instance for order \(668\). Compression and power spectral
density have been central to constructive searches for Legendre pairs
\cite{DjokovicKotsireas2015,KotsireasKoutschan2021,
KotsireasEtAl2023,KotsireasGomezGomezPerez2025,
KotsireasKoutschanWinterhof2025}. Constructive milestones include a Legendre
pair of length \(77\) and new pairs obtained through mod-\(3\) and mod-\(5\)
compression \cite{TurnerEtAl2021,KotsireasKoutschan2021,
KotsireasEtAl2023}. Multiplier and decimation actions provide a second
standard reduction: they replace \(333\) independent coordinates by signs on
multiplication orbits
\cite{KoukouvinosEtAl1997,GeorgiouKoukouvinos2002,TurnerEtAl2022}.

This paper determines how far that reduction can go at length \(333\). We use
``multiplier'' in the fixed sense: a subgroup
\(H\leq(\mathbb Z/333\mathbb Z)^\times\) fixes a sequence \(a\) when
\[
a_{hi}=a_i\qquad(h\in H,\ i\in\mathbb Z_{333}).
\]
The same subgroup is required to fix both members of the pair. This is
strictly narrower than a multiplier-with-translation relation. Our main result
is the following.

\begin{theorem}\label{thm:main}
Let \(H\leq(\mathbb Z/333\mathbb Z)^\times\). If there exists a Legendre pair
of length \(333\) whose two sequences are fixed by \(H\), then
\[
H\leq\ker\left((\mathbb Z/333\mathbb Z)^\times
\longrightarrow(\mathbb Z/3\mathbb Z)^\times\right)
\qquad\text{and}\qquad |H|\leq 6.
\]
More precisely, after the necessary mod-\(3\) reduction there are \(30\)
subgroups; \(21\) are impossible, including all \(19\) subgroups of order at
least \(9\). The nine subgroups not decided here have order at most
\(6\).
\end{theorem}

The principal new analytic ingredient closes the last order-\(9\) subgroup. Its
mod-\(37\) orbit structure forces a small, non-generic set of possible
column sums in the \(9\)-compression. The Legendre equations then force a
rigid pair of entries \(+17,-17\), whose displacement gives an immediate
autocorrelation contradiction. The same value-set compression also closes
one order-\(4\) and one order-\(6\) subgroup by exact finite enumeration.
The distinguishing point is that the compressed entries retain the exact
value set imposed by multiplier-orbit sizes; treating them as arbitrary odd
integers leaves the corresponding compressed system feasible.

The remaining high-order subgroups are eliminated by three complementary
mechanisms: integral Fourier values after mod-\(37\) compression, an exact
row-sum obstruction modulo \(24\), and orbit-level pseudo-Boolean systems.
For the last mechanism we provide proof-carrying evidence rather than relying
on solver status alone: five instances have independently checked DRAT traces,
and six have shorter direct pseudo-Boolean upper-bound certificates
\cite{WetzlerHeuleHunt2014}.

The theorem concerns fixed, untranslated common-multiplier symmetry. It leaves
the unrestricted length-\(333\) and order-\(668\) existence problems open.

The paper is organized as follows. Section~\ref{sec:preliminaries} recalls
Legendre pairs, Fourier duality, compression, and fixed multiplier symmetry.
Section~\ref{sec:reduction} gives the mod-\(3\) reduction and classifies its
subgroup lattice. Section~\ref{sec:value-compression} develops the value-set
\(9\)-compression and proves the analytic order-\(9\) obstruction.
Section~\ref{sec:certified} describes the exact orbit models and proof
certificates. Section~\ref{sec:classification} assembles the classification
and Section~\ref{sec:reproducibility} records the independently checkable
artifacts.

\section{Legendre pairs, compression, and fixed multipliers}
\label{sec:preliminaries}

All sequence indices are taken modulo the sequence length. For a real sequence
\(x=(x_0,\ldots,x_{L-1})\), its periodic autocorrelation is
\[
\operatorname{PAF}_x(s)=\sum_{j=0}^{L-1}x_jx_{j+s}.
\]
Let \(\zeta_L=\exp(2\pi i/L)\) and
\[
X(k)=\sum_{j=0}^{L-1}x_j\zeta_L^{jk}.
\]
The power spectral density is
\(\operatorname{PSD}_x(k)=|X(k)|^2\). The finite
Wiener--Khinchin identity is
\[
\operatorname{PSD}_x(k)
 =\sum_{s=0}^{L-1}\operatorname{PAF}_x(s)\zeta_L^{ks}.
\]

\begin{definition}\label{def:lp}
Two sequences \(a,b\in\{\pm1\}^L\), with \(L\) odd, form a
\emph{Legendre pair} if
\[
\operatorname{PAF}_a(s)+\operatorname{PAF}_b(s)=-2
\qquad(1\leq s<L).
\]
\end{definition}

The row-sum and Fourier formulations follow immediately.

\begin{lemma}\label{lem:lp-equivalences}
If \(a,b\) form a Legendre pair of odd length \(L\), then
\[
\sum_j a_j,\ \sum_j b_j\in\{\pm1\},
\]
and for every \(k\not\equiv0\pmod L\),
\[
\operatorname{PSD}_a(k)+\operatorname{PSD}_b(k)=2L+2.
\]
\end{lemma}

\begin{proof}
For every finite cyclic sequence,
\[
\sum_{s=0}^{L-1}\operatorname{PAF}_x(s)=\left(\sum_jx_j\right)^2.
\]
Summing the Legendre equations and using
\(\operatorname{PAF}_a(0)=\operatorname{PAF}_b(0)=L\) gives
\[
\left(\sum_ja_j\right)^2+\left(\sum_jb_j\right)^2=2.
\]
Both row sums are odd, so each is \(\pm1\). Fourier transformation of the
autocorrelation equations gives the stated PSD identity.
\end{proof}

After independently negating the sequences, their row sums may be normalized
to \(1\). Their negative supports then form supplementary difference sets with
parameters
\[
\operatorname{SDS}\left(L;\frac{L-1}{2},\frac{L-1}{2};
\frac{L-3}{2}\right).
\]
At \(L=333\) these are \(\operatorname{SDS}(333;166,166;165)\).
This connects the present problem directly to design-theoretic multiplier
methods \cite{GeorgiouKoukouvinos2002,KoukouvinosEtAl1997}.

\subsection{Compression}

Let \(L=dm\). The \(d\)-compression of \(x\) is the length-\(d\) integer
sequence
\[
\widetilde{x}_j=\sum_{t=0}^{m-1}x_{j+td},
\qquad 0\leq j<d.
\]
We use both the Fourier and autocorrelation forms of the standard compression
identity \cite{DjokovicKotsireas2015}.

\begin{lemma}[Compression identities]\label{lem:compression}
For \(0\leq k,s<d\),
\[
\widetilde{X}(k)=X(mk)
\]
and
\[
\operatorname{PAF}_{\widetilde{x}}(s)
 =\sum_{\substack{0\leq r<L\\r\equiv s\pmod d}}
  \operatorname{PAF}_x(r).
\]
\end{lemma}

\begin{proof}
Writing every coordinate uniquely as \(j+td\) gives
\[
X(mk)=\sum_{j=0}^{d-1}\sum_{t=0}^{m-1}
x_{j+td}\zeta_L^{(j+td)mk}
=\sum_{j=0}^{d-1}\widetilde{x}_j\zeta_d^{jk}.
\]
For the autocorrelation identity, expand
\(\operatorname{PAF}_{\widetilde{x}}(s)\). The pairs of summation indices are
in bijection with pairs \((i,r)\) where \(i\in\mathbb Z_L\) and
\(r\equiv s\pmod d\), giving the stated sum.
\end{proof}

\subsection{Fixed common multipliers}

Write \(U_L=(\mathbb Z/L\mathbb Z)^\times\).

\begin{definition}\label{def:hinvariant}
For \(H\leq U_L\), a sequence \(x\) is \emph{\(H\)-invariant} if
\[
x_{hi}=x_i\qquad(h\in H,\ i\in\mathbb Z_L).
\]
A Legendre pair is \(H\)-invariant when both of its sequences are
\(H\)-invariant.
\end{definition}

Thus an \(H\)-invariant sequence is constant on the multiplication orbits of
\(H\) on \(\mathbb Z_L\). Its PAF is constant on the corresponding shift
orbits, and its DFT is constant on the frequency orbits. This differs from
the broader condition
\[
x_{ti}=x_{i+c},
\]
which allows a translation. No assertion in this paper covers that broader
notion.

Global negation preserves every PAF, so in an orbit model we may normalize
the signs on the singleton orbit \(\{0\}\) independently for the two
sequences. Decimation classes and related multiplier actions have been
studied in \cite{TurnerEtAl2022}; here we use the subgroup lattice itself as
the index set for exact nonexistence results.

\section{The mod-3 reduction and the subgroup lattice}
\label{sec:reduction}

We first record the elementary arithmetic obstruction used repeatedly below.

\begin{lemma}\label{lem:not-two-squares}
The integer \(668\) is not a sum of two integer squares.
\end{lemma}

\begin{proof}
We have \(668=2^2\cdot167\), where \(167\) is prime and
\(167\equiv3\pmod4\). The sum-of-two-squares theorem therefore excludes a
representation \(668=u^2+v^2\).
\end{proof}

\subsection{The mod-3 obstruction}

\begin{proposition}\label{prop:mod3}
If an \(H\)-invariant Legendre pair of length \(333\) exists, then
\[
H\leq U_1:=
\ker\left(U_{333}\longrightarrow U_3\right)
=\{u\in U_{333}:u\equiv1\pmod3\}.
\]
\end{proposition}

\begin{proof}
Suppose \(H\) contains \(u\equiv2\pmod3\). The \(3\)-compression of either
sequence is invariant under multiplication by \(-1\) on \(\mathbb Z_3\), so
it has the form \((c_0,c_1,c_1)\). At a primitive cube root of unity its DFT
is the integer \(c_0-c_1\). By Lemma~\ref{lem:compression}, these two
compressed DFT values are the original DFT values at frequency \(111\).
Lemma~\ref{lem:lp-equivalences} would therefore express \(668\) as a sum of
two integer squares, contradicting Lemma~\ref{lem:not-two-squares}.
\end{proof}

This is the length-\(333\) instance of the compression and spectral
obstructions developed in
\cite{DjokovicKotsireas2015,KotsireasKoutschan2021,
KotsireasEtAl2023,KotsireasGomezGomezPerez2025}.

\subsection{The 30 compatible subgroups}

The Chinese remainder theorem gives
\[
U_{333}\cong U_9\times U_{37}\cong C_6\times C_{36},
\]
and hence
\[
U_1\cong C_3\times C_{36}
\cong C_4\times C_3\times C_9.
\]
There are exactly \(30\) subgroups of \(U_1\). Indeed, the Sylow
decomposition separates the \(C_4\) factor, which has three subgroups, from
\(C_3\times C_9\). The latter has one trivial subgroup, four subgroups of
order \(3\), four subgroups of order \(9\), and the full subgroup of order
\(27\), hence ten subgroups in total.

We use the stable identifiers in Table~\ref{tab:subgroups}. The generators
are residues modulo \(333\). The table also records the number \(r\) of
multiplication orbits on \(\mathbb Z_{333}\) and the orders \(h_9,h_{37}\)
of the two CRT images. These data are regenerated by closure of cyclic
subgroups and checked against the complete lattice.

\begin{table}[t]
\centering
\caption{Status by subgroup order inside the mod-$3$ kernel.}
\label{tab:order-summary}
\begin{tabular}{rrrr}
\toprule
$|H|$ & subgroups & impossible & open \\
\midrule
1 & 1 & 0 & 1 \\
2 & 1 & 0 & 1 \\
3 & 4 & 0 & 4 \\
4 & 1 & 1 & 0 \\
6 & 4 & 1 & 3 \\
9 & 4 & 4 & 0 \\
12 & 4 & 4 & 0 \\
18 & 4 & 4 & 0 \\
27 & 1 & 1 & 0 \\
36 & 4 & 4 & 0 \\
54 & 1 & 1 & 0 \\
108 & 1 & 1 & 0 \\
\bottomrule
\end{tabular}
\end{table}

\subsection{A full-image mod-37 obstruction}

\begin{proposition}\label{prop:mod37}
Suppose that the image of \(H\leq U_{333}\) in \(U_{37}\) is all of
\(U_{37}\). Then no \(H\)-invariant Legendre pair of length \(333\) exists.
\end{proposition}

\begin{proof}
The \(37\)-compression of either sequence is constant on the nonzero
residues and therefore has the form
\((c_0,c,\ldots,c)\). At any nonzero frequency its DFT is the integer
\[
c_0+c\sum_{j=1}^{36}\zeta_{37}^{jk}=c_0-c.
\]
By Lemma~\ref{lem:compression}, the corresponding original frequency is
\(9k\). The PSD identity would again express \(668\) as a sum of two
integer squares, contradicting Lemma~\ref{lem:not-two-squares}.
\end{proof}

This closes subgroup IDs \(25,26,27,29\).

\subsection{A row-sum residue obstruction}

Let \(\mathcal O_0,\ldots,\mathcal O_{r-1}\) be the multiplication orbits
of \(H\) on \(\mathbb Z_{333}\). The row sum of an invariant sequence is
\[
\sum_{q=0}^{r-1}|\mathcal O_q|\epsilon_q,
\qquad \epsilon_q\in\{\pm1\}.
\]
For IDs \(16,17,18,24\), exact subset-sum propagation modulo \(24\) gives
the same reachable set
\[
\{3,5,7,9,11,13,15,17,19,21\}\pmod{24}.
\]
It excludes both \(1\) and \(-1\), contradicting
Lemma~\ref{lem:lp-equivalences}.

\begin{proposition}\label{prop:rowsum}
No \(H\)-invariant Legendre pair of length \(333\) exists for subgroup IDs
\(16,17,18,24\).
\end{proposition}

\section{Value-set-restricted 9-compression}
\label{sec:value-compression}

The most useful new reduction occurs when \(H\) is trivial modulo \(9\).
Under
\[
\mathbb Z_{333}\cong\mathbb Z_9\times\mathbb Z_{37},
\]
each residue modulo \(9\) is a column of length \(37\). Let the image of
\(H\) in \(U_{37}\) have order \(h\). It has one orbit of size \(1\) and
\(36/h\) nonzero orbits of size \(h\). An invariant column sum therefore
lies in
\[
V_h=\left\{\epsilon_0+
h\sum_{j=1}^{36/h}\epsilon_j:
\epsilon_j\in\{\pm1\}\right\}.
\]
This orbit-size value set is the information lost by a free-odd
compression.

\begin{lemma}\label{lem:nine-compressed-system}
Let \((a,b)\) be an \(H\)-invariant Legendre pair of length \(333\), where
\(H\) is trivial modulo \(9\), and let
\(\widetilde a,\widetilde b\) be the \(9\)-compressions. Then
\[
\widetilde a_j,\widetilde b_j\in V_h,
\qquad
\sum_j\widetilde a_j,\sum_j\widetilde b_j\in\{\pm1\},
\]
\[
\operatorname{PAF}_{\widetilde a}(s)
+\operatorname{PAF}_{\widetilde b}(s)=-74
\qquad(1\leq s\leq8),
\]
and
\[
\sum_{j=0}^{8}\widetilde a_j^2+
\sum_{j=0}^{8}\widetilde b_j^2=594.
\]
\end{lemma}

\begin{proof}
The value-set statement follows from the orbit sizes in each column.
Compression preserves total sums. For \(s\neq0\pmod9\), exactly \(37\)
nonzero shifts modulo \(333\) are congruent to \(s\), so
Lemma~\ref{lem:compression} and the Legendre equations give \(-74\).
Finally,
\[
\sum_{s=0}^{8}\operatorname{PAF}_x(s)=\left(\sum_jx_j\right)^2.
\]
Adding this identity for the two compressed sequences and substituting the
eight nonzero correlations gives
\[
\operatorname{PAF}_{\widetilde a}(0)
+\operatorname{PAF}_{\widetilde b}(0)
=2-8(-74)=594.
\]
\end{proof}

\subsection{The final order-9 subgroup}

The only order-\(9\) subgroup not already closed by the certificates in
Section~\ref{sec:certified} is
\[
H_{12}=\langle10,46\rangle.
\]
It is trivial modulo \(9\), and its mod-\(37\) image has order \(9\).
Thus
\[
V_9=\{\pm1,\pm17,\pm19,\pm35,\pm37\}.
\]

\begin{theorem}\label{thm:id12}
There is no \(H_{12}\)-invariant Legendre pair of length \(333\).
\end{theorem}

\begin{proof}
By Lemma~\ref{lem:nine-compressed-system}, the \(18\) entries of
\(\widetilde a,\widetilde b\) have squares in
\[
\{1,289,361,1225,1369\}
\]
and total squared norm \(594\). Starting from \(18\) entries of square \(1\),
the excess is \(576\). Entries of square \(1225\) or \(1369\) are already
too large, while the remaining count equation is
\[
288x+360y=576,
\]
or \(4x+5y=8\). Its unique nonnegative solution is
\((x,y)=(2,0)\). Hence exactly two compressed entries have absolute value
\(17\), and all other entries have absolute value \(1\).

The two large entries cannot be split between the sequences: a length-\(9\)
sequence with one entry of absolute value \(17\) and eight entries of
absolute value \(1\) has absolute row sum at least \(9\), contrary to the
required row sum \(\pm1\). Thus, after interchanging the sequences, one
compressed sequence contains \(+17\) at a position \(p\) and \(-17\) at a
distinct position \(q\), together with seven entries in \(\{\pm1\}\); the
other sequence is entirely in \(\{\pm1\}\).

Set \(s=q-p\pmod9\). Because \(s\neq0\) and \(2s\not\equiv0\pmod9\), the nine
terms in \(\operatorname{PAF}_{\widetilde a}(s)\) comprise one big--big
term, two big--small terms, and six small--small terms. Consequently,
\[
\operatorname{PAF}_{\widetilde a}(s)
\leq -17^2+2\cdot17+6=-249.
\]
The second compressed sequence has nine \(\{\pm1\}\)-entries, so its PAF is
at most \(9\). Therefore
\[
\operatorname{PAF}_{\widetilde a}(s)
+\operatorname{PAF}_{\widetilde b}(s)
\leq -240<-74,
\]
contradicting Lemma~\ref{lem:nine-compressed-system}.
\end{proof}

The proof uses only one compressed shift and no search. An independent
enumeration provides redundant confirmation: the \(5{,}292\) admissible
compressed sequences consist of \(252\) sequences with no large entry and
\(5{,}040\) sequences with two large entries, giving
\(2\cdot252\cdot5{,}040=2{,}540{,}160\) ordered pairs to test. CP-SAT and SMT
encodings provide further independent confirmation.

\subsection{Two lower-order subgroups}

The same lemma applies to every subgroup trivial modulo \(9\). For the
order-\(4\) and order-\(6\) images, the value sets are
\[
V_4=\{\pm3,\pm5,\pm11,\pm13,\pm19,\pm21,\pm27,\pm29,\pm35,\pm37\},
\]
\[
V_6=\{\pm1,\pm11,\pm13,\pm23,\pm25,\pm35,\pm37\}.
\]
Exact square-sum-pruned enumeration gives the following.

\begin{table}[t]
\centering
\caption{Exact value-set \(9\)-compression decisions. The candidate count is
the number of row-sum and squared-norm compatible compressed sequences visited
by the exact enumerator.}
\label{tab:value-compression}
\begin{tabular}{rrrrr}
\toprule
ID & \(h\) & \(|V_h|\) & square multisets & candidates \\
\midrule
6  & 4 & 20 & 9 & \(2{,}428{,}992\) \\
8  & 6 & 14 & 1 & \(148{,}428\) \\
12 & 9 & 10 & 1 & \(5{,}292\) \\
\bottomrule
\end{tabular}
\end{table}

\begin{proposition}\label{prop:id6-id8}
No \(H\)-invariant Legendre pair of length \(333\) exists for subgroup IDs
\(6\) or \(8\).
\end{proposition}

\begin{proof}
For each subgroup, the enumerator first lists all multisets of \(18\) squares
from \(V_h\) summing to \(594\). It then generates every length-\(9\) sequence
with row sum \(\pm1\) compatible with such a multiset and hashes its four
independent PAF values. No two profiles sum to
\((-74,-74,-74,-74)\). The counts in
Table~\ref{tab:value-compression} are reproduced by a dependency-free
verifier. For ID \(8\), an independent CP-SAT model also returns infeasible.
For ID \(6\), CP-SAT did not terminate within its budget; its exclusion rests
on the complete enumeration, which is independently reproduced by the
dependency-free verifier.
\end{proof}

For comparison, IDs \(0\) and \(1\) have the full odd value set
\(\{-37,-35,\ldots,35,37\}\); a stored witness satisfies the compressed
system, so this relaxation is genuinely feasible. ID \(3\) has a smaller
value set but remains undecided by the present enumeration budget.

\section{Exact orbit models and proof-carrying exclusions}
\label{sec:certified}

We now describe the exact finite systems used for the remaining subgroups. Let
\(\mathcal O_0,\ldots,\mathcal O_{r-1}\) be the multiplication orbits of
\(H\) on \(\mathbb Z_{333}\), with \(\mathcal O_0=\{0\}\), and let
\(x_q\in\{\pm1\}\) be the common sequence value on \(\mathcal O_q\).
For a shift \(s\), define
\[
D_s(q,r)=\#\{i\in\mathcal O_q:i+s\in\mathcal O_r\},
\]
\[
c_s=\sum_qD_s(q,q),\qquad
W_s(q,r)=D_s(q,r)+D_s(r,q)\quad(q<r).
\]
Then
\begin{equation}\label{eq:orbit-paf}
\operatorname{PAF}_x(s)=
c_s+\sum_{q<r}W_s(q,r)x_qx_r.
\end{equation}
Only one shift representative from each multiplication orbit is needed. The
row-sum condition is
\begin{equation}\label{eq:orbit-row}
\sum_q|\mathcal O_q|x_q\in\{\pm1\}.
\end{equation}
Equations~\eqref{eq:orbit-paf} and \eqref{eq:orbit-row} are checked both
against direct length-\(333\) arithmetic and by an independent implementation.

\subsection{Pseudo-Boolean formulation}

Write \(x_q=1-2z_q\), where \(z_q\) is Boolean, and introduce
\(w_{qr}=z_q\mathbin{\mathsf{xor}}z_r\). Then
\(x_qx_r=1-2w_{qr}\). Because
\[
c_s+\sum_{q<r}W_s(q,r)=333,
\]
the two-sequence PAF equation becomes the nonnegative weighted equality
\begin{equation}\label{eq:pb}
\sum_{q<r}W_s(q,r)
\left(w^{(a)}_{qr}+w^{(b)}_{qr}\right)=334.
\end{equation}
The row-sum condition becomes
\[
\sum_q|\mathcal O_q|z_q\in\{166,167\}.
\]

\begin{lemma}[Encoding equivalence]\label{lem:encoding-equivalence}
For every subgroup \(H\) encoded in this section, the archived CNF
\(\Phi_H\) is satisfiable if and only if an \(H\)-invariant Legendre pair of
length \(333\) exists. The unit-split CNF used for DRAT checking is
equisatisfiable with \(\Phi_H\).
\end{lemma}

\begin{proof}
An \(H\)-invariant sequence is determined by one sign on each multiplication
orbit. Independent global negation of either sequence preserves every PAF and
the row-sum domain \(\{\pm1\}\), so both signs on the orbit \(\{0\}\) may be
normalized to \(+1\). With \(z_q=(1-x_q)/2\), the row sums are exactly the two
weighted constraints above. Moreover,
\(\operatorname{PAF}_x(hs)=\operatorname{PAF}_x(s)\) for \(h\in H\), so one
representative from each nonzero shift orbit imposes all Legendre equations.
Equation~\eqref{eq:orbit-paf} converts each of them exactly into
\eqref{eq:pb}. The four Tseitin clauses define each XOR in both directions,
and the carry-save and ripple full-adder clauses define every bit of each
weighted sum. Thus an invariant Legendre pair extends to a satisfying
assignment of \(\Phi_H\), while any satisfying assignment restricts to orbit
signs satisfying every row-sum and PAF equation. Finally, replacing a unit
clause \((\ell)\) by
\((\ell\vee e)\wedge(\ell\vee\neg e)\), with \(e\) fresh, preserves
existential satisfiability.
\end{proof}

The implementation is tested against Lemma~\ref{lem:encoding-equivalence},
not assumed to realize it. A standalone standard-library generator independently
reconstructs the subgroup orbits and PAF matrices, builds a separate
XOR/adder encoding, checks its arithmetic semantics on sampled primary
assignments, and reproduces the SHA-256 hashes of both DIMACS variants. The
primary generator is also exhaustively audited on singleton-orbit lengths
\(5\) and \(7\), where it accepts directly verified positive Legendre-pair
controls.

For IDs \(11,15,19,23,24,28\), at shift \(111\) the left side of
\eqref{eq:pb} has maximum \(222\), producing a direct arithmetic
contradiction. For IDs \(13,14,20,21,22\), the complete system is encoded in
CNF using exact Tseitin XOR clauses and carry-save weighted adders. The saved
DRAT traces are independently checked with \texttt{drat-trim}
\cite{WetzlerHeuleHunt2014}. The proof inputs, traces, checker source, dependency
pins, and positive controls are archived with the paper.

\begin{table}[t]
\centering
\caption{Proof-carrying pseudo-Boolean instances. Variable and clause counts refer to the unit-split CNF checked by \texttt{drat-trim}; proof size is shown only for nontrivial DRAT traces.}
\label{tab:proof-artifacts}
\small
\begin{tabular}{rrrrlr}
\toprule
ID & $r$ & variables & clauses & evidence & MiB \\
\midrule
11 & 65 & 68,112 & 459,592 & direct PB & -- \\
13 & 41 & 38,680 & 262,778 & DRAT & 51.8 \\
14 & 41 & 38,680 & 262,778 & DRAT & 48.5 \\
15 & 50 & 39,411 & 264,856 & direct PB & -- \\
19 & 35 & 21,264 & 142,720 & direct PB & -- \\
20 & 27 & 7,870 & 50,942 & DRAT & 3.0 \\
21 & 23 & 11,094 & 74,390 & DRAT & 1043.8 \\
22 & 23 & 11,094 & 74,390 & DRAT & 14.4 \\
23 & 25 & 10,456 & 69,600 & direct PB & -- \\
24 & 20 & 6,258 & 41,236 & direct PB & -- \\
28 & 15 & 2,786 & 17,828 & direct PB & -- \\
\bottomrule
\end{tabular}
\end{table}

\subsection{Meet-in-the-middle cross-checks}

For orbit counts \(r\leq27\), a separate meet-in-the-middle implementation
enumerates all \(2^{r-1}\) normalized orbit assignments by Gray code. It keeps
assignments with row sum \(\pm1\), stores their PAF profiles, and searches for
complementary profiles summing to \(-2\) in every shift orbit. It was used as a
cross-check for IDs \(20,21,22,23,24,28\). The archived source and orbit-spec
files permit these runs to be repeated, but the MITM output records are not
part of the immutable proof bundle. Accordingly, these computations are not
used as the archival proof vehicle; the checked DRAT traces, direct
pseudo-Boolean bounds, and row-sum certificate provide that evidence.

\section{Classification}\label{sec:classification}

We can now assemble the preceding exclusions.

\begin{theorem}\label{thm:classification}
Inside the mod-\(3\) kernel \(U_1\), exactly \(21\) of the \(30\) fixed
common-multiplier subgroups are proved impossible by the methods of
this paper. All \(19\) subgroups of order at least \(9\) are impossible.
The unresolved IDs are
\[
0,1,2,3,4,5,7,9,10,
\]
with orders \(1,2,3,3,3,3,6,6,6\), respectively.
\end{theorem}

\begin{proof}
Proposition~\ref{prop:mod37} closes IDs \(25,26,27,29\), and
Proposition~\ref{prop:rowsum} closes IDs \(16,17,18,24\).
Theorem~\ref{thm:id12} and Proposition~\ref{prop:id6-id8} close IDs
\(12\) and \(6,8\). The direct pseudo-Boolean bound closes IDs
\(11,15,19,23,28\) (and independently ID \(24\)); the checked DRAT proofs close
IDs \(13,14,20,21,22\). These are precisely the \(21\) impossible rows in
Table~\ref{tab:subgroups}. The remaining nine rows are not declared feasible;
they are only not decided by the present methods.
\end{proof}

\begin{proof}[Proof of Theorem~\ref{thm:main}]
If \(H\not\leq U_1\), Proposition~\ref{prop:mod3} rules out an
\(H\)-invariant pair. If \(H\leq U_1\) and \(|H|\geq9\),
Theorem~\ref{thm:classification} rules it out. Therefore any fixed common
multiplier subgroup of a Legendre pair of length \(333\) has order at most
\(6\).
\end{proof}

\section{Reproducibility and independent verification}
\label{sec:reproducibility}

The complete source, certificates, and verification instructions are archived
at Zenodo \cite{Ramos2026Artifacts}. The archive records:
\begin{itemize}
\item the full subgroup lattice and orbit data;
\item the row-sum and compression certificates;
\item meet-in-the-middle source and orbit specifications for independent
reruns;
\item CNF inputs and DRAT traces;
\item direct pseudo-Boolean upper-bound certificates;
\item positive Legendre-pair controls of lengths \(5\) and \(7\);
\item an independent rebuild of orbit matrices and deterministic CNF
serialization; and
\item SHA-256 manifests and a bogus-proof rejection test.
\end{itemize}

The stable subgroup IDs and verdicts are recorded in the top-level
classification. Some phase-one JSON records retain the computational method by
which an instance was first closed (CP-SAT or meet-in-the-middle). For
proof-vehicle attribution, the authoritative record is
\texttt{proof\_phase2/manifest.json}: it identifies the final direct
pseudo-Boolean and DRAT certificates used in this paper.

The analytic proof of Theorem~\ref{thm:id12} and the value-set verifier require
only the Python standard library. The proof-carrying SAT archive contains the
portable \texttt{drat-trim} source and reruns every proof check. No
nonexistence statement in Theorem~\ref{thm:classification} is based on a
timeout or a floating-point comparison.

\section{Discussion and limitations}\label{sec:discussion}

The result shows that strong fixed common-multiplier symmetry cannot solve the
length-\(333\) Legendre-pair problem: every such group of order at least \(9\)
is excluded. This is useful computationally because multiplier invariance is
one of the most effective reductions in searches for cyclic combinatorial
objects. Any successful search at length \(333\) must now focus on the nine
weak-symmetry subgroups in Table~\ref{tab:subgroups}, a pair with only
translation-twisted multiplier behavior, or a pair with no useful common
multiplier at all.

The id12 proof also illustrates a general principle. Compression is stronger
when the compressed entries are not treated as arbitrary odd integers but are
restricted by the exact orbit sizes of the multiplier image. In the present
case the free-odd \(9\)-compression has \(842\) possible square multisets and
is feasible, whereas the order-\(9\) mod-\(37\) image permits one square
multiset and yields a one-shift contradiction. Similar value-set restrictions
may be useful at other composite lengths.

Nine low-order subgroups remain open, and Legendre pairs are only one route to
Hadamard order \(668\). Thus the unrestricted existence question is unchanged.

\section{Conclusion}\label{sec:conclusion}

We constrained fixed common-multiplier symmetry for Legendre pairs of length
\(333\) to nine residual low-order subgroups. Of the \(30\) subgroups surviving
the necessary mod-\(3\) reduction, \(21\) are impossible, including all
subgroups of order at least \(9\). The final order-\(9\) subgroup admits a short
analytic contradiction based on an orbit-size-restricted \(9\)-compression.
The computational remainder is backed by exact arithmetic, independent
models, and checkable proof certificates. Consequently, any
multiplier-invariant Legendre-pair route to Hadamard order \(668\) must use a
fixed common multiplier group of order at most \(6\).

\renewcommand{\thetable}{A\arabic{table}}
\setcounter{table}{0}
\section*{Appendix A. Complete subgroup ledger}\label{app:subgroups}
\addcontentsline{toc}{section}{Appendix A. Complete subgroup ledger}

The IDs in Table~\ref{tab:subgroups} are stable across the source, certificates,
and proof archive. Generators are residues modulo \(333\).

\begin{center}
\centering
\captionof{table}{The 30 subgroups of the mod-$3$ kernel, in the stable numbering used by the archived computation. Here $r$ is the number of multiplication orbits on $\mathbb Z_{333}$, and $h_9,h_{37}$ are the orders of the images modulo $9$ and $37$. ``Open'' means not decided by the methods in this paper, not that an invariant Legendre pair is known.}
\label{tab:subgroups}
\small
\setlength{\tabcolsep}{3.6pt}
\begin{tabular}{r l r r r r l}
\toprule
ID & generators & $|H|$ & $r$ & $h_9$ & $h_{37}$ & strongest certificate \\
\midrule
0 & $\{1\}$ & 1 & 333 & 1 & 1 & Open \\
1 & $\langle 73\rangle$ & 2 & 171 & 1 & 2 & Open \\
2 & $\langle 112\rangle$ & 3 & 185 & 3 & 1 & Open \\
3 & $\langle 10\rangle$ & 3 & 117 & 1 & 3 & Open \\
4 & $\langle 121\rangle$ & 3 & 113 & 3 & 3 & Open \\
5 & $\langle 211\rangle$ & 3 & 113 & 3 & 3 & Open \\
6 & $\langle 73,154\rangle$ & 4 & 90 & 1 & 4 & Value-set $9$-compression \\
7 & $\langle 73,112\rangle$ & 6 & 95 & 3 & 2 & Open \\
8 & $\langle 10,64\rangle$ & 6 & 63 & 1 & 6 & Value-set $9$-compression \\
9 & $\langle 73,85\rangle$ & 6 & 59 & 3 & 6 & Open \\
10 & $\langle 73,121\rangle$ & 6 & 59 & 3 & 6 & Open \\
11 & $\langle 10,112\rangle$ & 9 & 65 & 3 & 3 & Direct PB upper bound \\
12 & $\langle 10,46\rangle$ & 9 & 45 & 1 & 9 & Value-set $9$-compression \\
13 & $\langle 7\rangle$ & 9 & 41 & 3 & 9 & Checked DRAT proof \\
14 & $\langle 10,16\rangle$ & 9 & 41 & 3 & 9 & Checked DRAT proof \\
\midrule
15 & $\langle 31\rangle$ & 12 & 50 & 3 & 4 & Direct PB upper bound \\
16 & $\langle 10,64,82\rangle$ & 12 & 36 & 1 & 12 & Row-sum obstruction modulo $24$ \\
17 & $\langle 73,85,88\rangle$ & 12 & 32 & 3 & 12 & Row-sum obstruction modulo $24$ \\
18 & $\langle 73,121,154\rangle$ & 12 & 32 & 3 & 12 & Row-sum obstruction modulo $24$ \\
19 & $\langle 10,64,85\rangle$ & 18 & 35 & 3 & 6 & Direct PB upper bound \\
20 & $\langle 10,28\rangle$ & 18 & 27 & 1 & 18 & MITM and checked DRAT proof \\
21 & $\langle 7,58\rangle$ & 18 & 23 & 3 & 18 & MITM and checked DRAT proof \\
22 & $\langle 4\rangle$ & 18 & 23 & 3 & 18 & MITM and checked DRAT proof \\
23 & $\langle 7,16\rangle$ & 27 & 25 & 3 & 9 & Direct PB upper bound \\
24 & $\langle 10,31\rangle$ & 36 & 20 & 3 & 12 & Row-sum obstruction modulo $24$ \\
25 & $\langle 10,19\rangle$ & 36 & 18 & 1 & 36 & Surjective mod-$37$ compression \\
26 & $\langle 4,13\rangle$ & 36 & 14 & 3 & 36 & Surjective mod-$37$ compression \\
27 & $\langle 7,22\rangle$ & 36 & 14 & 3 & 36 & Surjective mod-$37$ compression \\
28 & $\langle 4,7\rangle$ & 54 & 15 & 3 & 18 & Direct PB upper bound \\
29 & $\langle 4,7,13\rangle$ & 108 & 10 & 3 & 36 & Surjective mod-$37$ compression \\
\bottomrule
\end{tabular}
\end{center}

\section*{Acknowledgments}

The authors thank the developers of SageMath, OR-Tools, CaDiCaL, and
\texttt{drat-trim}. Generative AI tools were used to assist with code
generation, literature discovery, and language editing. The authors
independently checked the mathematical arguments, citations, computations, and
proof certificates and take full responsibility for the manuscript.

\section*{Data availability statement}

Source code, generated certificates, CNF instances, DRAT proofs, independent
verifiers, and exact reproduction instructions are publicly archived at
\url{https://doi.org/10.5281/zenodo.21498698}. The corresponding source
repository is
\url{https://github.com/Arthur742Ramos/hadamard-668-multiplier-obstructions}.

\section*{Funding information}

This research received no external funding.

\section*{Conflict of interest}

The authors declare no conflict of interest.

\section*{Author contributions}

Arthur F. Ramos: conceptualization, methodology, software, formal analysis,
investigation, validation, writing -- original draft, and artifact curation.
David B. Hulak: validation, and writing -- review and editing.
Ruy J. G. B. de Queiroz: supervision, and writing -- review and editing.

\bibliography{references}

\begin{thebibliography}{10}
\providecommand{\url}[1]{\texttt{#1}}
\providecommand{\urlprefix}{URL }
\expandafter\ifx\csname urlstyle\endcsname\relax
  \providecommand{\doi}[1]{doi:\discretionary{}{}{}#1}\else
  \providecommand{\doi}{doi:\discretionary{}{}{}\begingroup
  \urlstyle{rm}\Url}\fi

\bibitem{CatiPasechnik2024}
M.~Cati and D.~V. Pasechnik, \textit{A database of constructions of {H}adamard
  matrices}, arXiv:2411.18897  (2024), \doi{10.48550/arXiv.2411.18897}.

\bibitem{DjokovicKotsireas2015}
D.~{\v Z}. Djokovi{\'c} and I.~S. Kotsireas, \textit{Compression of periodic
  complementary sequences and applications}, Designs, Codes and Cryptography
  \textbf{74} (2015), no.~2, 365--377. ArXiv:1302.0571.

\bibitem{Eliahou2025}
S.~Eliahou, \textit{A 64-modular {H}adamard matrix of order 668}, Australasian
  Journal of Combinatorics \textbf{93} (2025), no.~2, 422--427.

\bibitem{FletcherGysinSeberry2001}
R.~J. Fletcher, M.~Gysin, and J.~Seberry, \textit{Application of the discrete
  {F}ourier transform to the search for generalised {L}egendre pairs and
  {H}adamard matrices}, Australasian Journal of Combinatorics \textbf{23}
  (2001), 75--86.

\bibitem{GeorgiouKoukouvinos2002}
S.~Georgiou and C.~Koukouvinos, \textit{On generalized {L}egendre pairs and
  multipliers of the corresponding supplementary difference sets}, Utilitas
  Mathematica \textbf{61} (2002), 47--63. Zbl 1001.05031.

\bibitem{Hadamard1893}
J.~Hadamard, \textit{R{\'e}solution d'une question relative aux
  d{\'e}terminants}, Bulletin des Sciences Math{\'e}matiques \textbf{17}
  (1893), 240--246.

\bibitem{KharaghaniTayfehRezaie2005}
H.~Kharaghani and B.~Tayfeh-Rezaie, \textit{A {H}adamard matrix of order 428},
  Journal of Combinatorial Designs \textbf{13} (2005), no.~6, 435--440.

\bibitem{KotsireasGomezGomezPerez2025}
I.~S. Kotsireas, A.~I. G{\'o}mez, and D.~G{\'o}mez-P{\'e}rez, \textit{On
  properties of {L}egendre pairs under compression}, \textit{Proceedings of the
  2025 International Symposium on Symbolic and Algebraic Computation}, ACM,
  2025, 79--86, \doi{10.1145/3747199.3747549}.

\bibitem{KotsireasKoutschan2021}
I.~S. Kotsireas and C.~Koutschan, \textit{Legendre pairs of lengths $\ell
  \equiv 0 \pmod 3$}, Journal of Combinatorial Designs \textbf{29} (2021),
  no.~12, 870--887. ArXiv:2101.03116.

\bibitem{KotsireasKoutschanWinterhof2025}
I.~S. Kotsireas, C.~Koutschan, and A.~Winterhof, \textit{Quaternary {L}egendre
  pairs {II}}, Discrete Mathematics \textbf{348} (2025), no.~9, 114501. Article
  114501; arXiv:2408.16318.

\bibitem{KotsireasEtAl2023}
I.~S. Kotsireas, C.~Koutschan, D.~A. Bulutoglu, D.~M. Arquette, J.~S. Turner,
  and K.~J. Ryan, \textit{Legendre pairs of lengths $\ell \equiv 0 \pmod 5$},
  Special Matrices \textbf{11} (2023), no.~1, 20230105. Article 20230105;
  arXiv:2111.02105.

\bibitem{KoukouvinosEtAl1997}
C.~Koukouvinos, J.~Seberry, A.~L. Whiteman, and M.-y. Xia, \textit{Optimal
  designs, supplementary difference sets and multipliers}, Journal of
  Statistical Planning and Inference \textbf{62} (1997), no.~1, 81--90.

\bibitem{Paley1933}
R.~E. A.~C. Paley, \textit{On orthogonal matrices}, Journal of Mathematics and
  Physics \textbf{12} (1933), 311--320.

\bibitem{Ramos2026Artifacts}
A.~Ramos, \textit{Multiplier obstructions for {L}egendre pairs of length 333},
  Zenodo  (2026), \doi{10.5281/zenodo.21498698}. Version 1.0.0.

\bibitem{TurnerEtAl2021}
J.~S. Turner, I.~S. Kotsireas, D.~A. Bulutoglu, and A.~J. Geyer, \textit{A
  {L}egendre pair of length 77 using complementary binary matrices with fixed
  marginals}, Designs, Codes and Cryptography \textbf{89} (2021), no.~6,
  1321--1333. ArXiv:2101.10918.

\bibitem{TurnerEtAl2022}
J.~S. Turner, D.~A. Bulutoglu, D.~Baczkowski, and A.~J. Geyer, \textit{Counting
  the decimation classes of binary vectors with relatively prime length and
  density}, Journal of Algebraic Combinatorics \textbf{55} (2022), no.~1,
  61--87.

\bibitem{WetzlerHeuleHunt2014}
N.~Wetzler, M.~J.~H. Heule, and W.~A. Hunt, \textit{{DRAT}-trim: Efficient
  checking and trimming using expressive clausal proofs}, \textit{Theory and
  Applications of Satisfiability Testing -- SAT 2014}, \textit{Lecture Notes in
  Computer Science}, vol. 8561, Springer, 2014, 422--429,
  \doi{10.1007/978-3-319-09284-3_31}.

\end{thebibliography}

\end{document}